\documentclass{amsart}

\usepackage[dvipdfm]{graphicx}
\usepackage{epstopdf}
\usepackage{epsf,latexsym,amsmath,amsfonts,amssymb,subfigure,mathrsfs,stmaryrd,eufrak,esint,amscd}
\usepackage[notref,notcite]{showkeys}
\usepackage{color}

\theoremstyle{plain}
\newtheorem{theorem}{Theorem}[section]
\newtheorem{lemma}[theorem]{Lemma}

\newtheorem{coro}[theorem]{Corollary}

\theoremstyle{definition}

\theoremstyle{remark}

\numberwithin{equation}{section}
\newcommand{\abs}[1]{\lvert#1\rvert}

\newcommand{\Lr}[1]{\left(#1\right)}
\newcommand{\lr}[1]{\bigl(#1\bigr)}
\newcommand{\set}[2]{\{\,#1\,\mid\,#2\,\}}

\newcommand{\nm}[2]{\|\,#1\,\|_{#2}}
\newcommand{\jump}[1]{[\![#1]\!]}

\newcommand{\wnm}[1]{|\!|\!|#1|\!|\!|}

\newcommand{\mc}[1]{\mathcal{#1}}
\newcommand{\mb}[1]{\mathbb{#1}}
\newcommand{\ov}[1]{\overline{#1}}

\def\na{\nabla}

\def\veps{\varepsilon}
\def\pa{\partial}
\def\Om{\Omega}
\def\om{\omega}
\def\lam{\lambda}

\def\md{\mathrm{d}}

\def\T{\mathrm{T}}

\newcommand{\nn}{\nonumber}
\begin{document}
\title[The TRUNC-type element]{The TRUNC element in any dimension and application to a modified Poisson equation}

\author{Hongliang Li}
\address{Department of Mathematics, Sichuan Normal University, Chengdu 610066, China\\
 email: lhl@sicnu.edu.cn;lihongliang@lsec.cc.ac.cn}

\author{Pingbing Ming}
\address{LSEC, Institute of Computational Mathematics and Scientific/Engineering Computing, AMSS\\
 Chinese Academy of Sciences, No. 55, East Road Zhong-Guan-Cun, Beijing 100190, China\\
 and School of Mathematical Sciences, University of Chinese Academy of Sciences, Beijing 100049, China\\
 email: mpb@lsec.cc.ac.cn}
 
\author{Yinghong Zhou}
\address{Department of Mathematics, Sichuan Normal University, Chengdu 610066, China\\}
\begin{abstract}
We introduce a novel TRUNC finite element in $n$ dimensions, encompassing the traditional TRUNC triangle as a particular instance. By establishing the weak continuity identity, we identify it as crucial for error estimate. This element is utilized to approximate a modified Poisson equation defined on a convex polytope, originating from the nonlocal electrostatics model. We have substantiated a uniform error estimate and conducted numerical tests on both the smooth solution and the solution with a sharp boundary layer, which align with the theoretical predictions.
\end{abstract}
\date{\today}
\maketitle
\section{Introduction}
The development of an efficient finite element approximation for fourth-order elliptic equations presents an intriguing challenge. The use of $H^2$-conforming finite elements appears to be a natural approach for approximating the fourth-order problem, but such elements necessitate global $C^1$ continuity, which can only be achieved with polynomials of sufficiently high degree. Furthermore, the implementation of $H^2$-conforming finite elements may pose difficulties due to the large number of degrees of freedom, especially in three dimensions. It is well-established that at least ninth-degree polynomials are required to construct a conforming finite element space on a tetrahedral mesh, demanding a minimum of $220$ degrees of freedom on each element, including the fourth-order derivative as a degree of freedom~\cite{Zenik:1973,Zhang:2009}. 

To address the challenge of attaining $C^1$continuity, a common approach involves employing nonconforming finite elements. In two dimensions, several successful elements have been developed, such as the Morley triangle~\cite{Morley:1968}, the TRUNC triangle~\cite{Argyris:1980}, and the Zienkiewicz-type triangle~\cite{Zienkiewicz:1965,Specht:1988,LiMingShi:2020}, among others. For recent advancements in nonconforming elements, particularly in $n$ dimensions, we refer to ~\cite{WangXu:2006,WangShiXu:2007,GuzmanLeykekhmanNeilan:2012,Wu:2019, HuZhang:2020_1,HuZhang:2020_2} and the references therein.

The TRUNC element stands out among available finite elements for its utilization of the same shape functions as the Zienkiewicz triangle and its convergence on arbitrary meshes, in contrast to the specific mesh requirements of the Zienkiewicz triangle, which converges only on meshes satisfying three parallel line conditions~\cite{Shi:1987}. Notably, the TRUNC triangle relies solely on function values and first-order derivatives at each vertex as degrees of freedom, without incorporating any edge degrees of freedom~\cite{Walk:2014}, making it particularly attractive for the implementation.

Furthermore, numerical benchmarks indicate that the accuracy of the TRUNC triangle is comparable to that of quintic conforming Argyris-Bell triangles~\cite{Argyris:1982,Argyris:1986} despite utilizing cubic polynomials in its shape functions. It is worth noting that the TRUNC triangle can also be viewed as a variant of Bergan's free formulation scheme~\cite{Bergan:1984}. The unique numerical integration technique employed in the TRUNC triangle sets it apart, leading to a convergence proof that deviates from the standard framework of nonconforming finite elements~\cite{Berge:1972}. The first convergence proof of the TRUNC triangle for the plate bending problem was achieved by \textsc{Shi}~\cite{Shi:1987}, with further extensions discussed in ~\cite{ShiZhang:1990,ChenShi:1991,GuoHuangShi:2006}. 

The aim of this study is to generalize the TRUNC triangle to higher dimensions, as the TRUNC triangle naturally allows for such an extension. This extension is particularly attractive in three dimensions, as it eliminates the need for edge or face degrees of freedom. As stated in~\cite{Walk:2014}, {\em this eliminates the need to associate a basis for the planes perpendicular to each edge; such a basis cannot depend continuously upon the edge orientation.}  

As an application of the TRUNC element, we consider a modified Poisson equation posed over a convex polytope $\Om$.
\begin{equation}\label{eq:bc}
\left\{
\begin{aligned}
\veps^2\triangle^2u-\triangle u&=f\qquad&&\text{in}\quad\Om,\\
u=\pa_nu&=0\qquad&&\text{on}\quad\pa\Om,
\end{aligned}\right.
\end{equation}
where $n$ is the outward unit normal of $\pa\Om$. 
The modified Poisson equation~\eqref{eq:bc} emerges in the realm of nonlocal electrostatics, particularly concerning the dielectric characteristics of ionic liquids and its efficacy in predicting the configuration of the electrical double layer~\cite{Hildebrandt:2004,Bazant:2011}. The filed $u$ in this model signifies the nonlocal electrostatic potential, while $\veps$ denotes the electrostatic correlation length parameter, assumed to be small. As $\veps$ approaches zero, the modified Poisson equation~\eqref{eq:bc} converges to the standard Poisson electrostatics. Our focus in this study lies within the regime when $\veps$ tends to zero.

In general, nonlocal electrostatics postulates a connection between the dielectric displacement field and the electric field, mediated by a permittivity kernel contingent on two spatial parameters, represented by
\[
\mathbf{D}(x)=\int\epsilon_{\veps}(y,x)\na u\,\md y,
\]
where $\mathbf{D}$ is the displacement field and $\epsilon_{\veps}(y,x)$ is the dielectric permittivity tensor. The modified Poisson equation~\eqref{eq:bc} may be viewed as a particular case of nonlocal electrostatics specifying with a Yukawa-type permittivity kernel tensor. \textsc{Li and Ming}~\cite{LiMing:2024} have proposed an asymptotic-preserving finite element method for~\eqref{eq:bc}, which completely preserves the asymptotic transition of the underlying partial differential equation. \textsc{Tian and Du}~\cite{TianDu:2014,TianDu:2020} have developed asymptotically compatible schemes for such models so as to offer robust numerical discretization to problems involving the nonlocal interactions on multiple scales. 

The convergence proof of the TRUNC element approximation of~\eqref{eq:bc} relies on two key components. The first component is a weak continuity identity, which can be viewed as a natural extension of two orthogonal identities established in ~\cite{Shi:1987}. The second component involves a novel regularity result, previously demonstrated only for problems posed on a bounded convex polygon. In our study, we extend this result to a convex polytope by integrating insights from recent investigations on the strain gradient elasticity model~\cite{LiMingWang:2021} and leveraging the combined effect of singular perturbation and 
homogenization~\cite{Niu:2022}. 

In this paper we shall adopt the standard notations for Sobolev spaces and norms~\cite{AdamsFournier:2003}. We denote $L^2(\Om)$ the square integrable function space over $\Om$, which is equipped with norm $\nm{\cdot}{L^2(\Om)}$ and inner product $(\cdot,\cdot)$. Let $H^m(\Om)$ be the Sobolev space of square integrable function with $m$-th weak derivatives, which is equipped with norm $\nm{\cdot}{H^m(\Om)}$. We may drop $\Om$ in the norm when there is no confusion may occur. We denote $H_0^m(\Om)$ the closure in $H^m(\Om)$ of the space of $C^\infty(\Om)$ functions with compact supports in $\Om$.

The subsequent sections of the paper are structured as follows: In Section 2, we present the formulation of the TRUNC element.
Section 3 contains the derivation of a uniform error estimate concerning $\veps$ for a modified Poisson problem.
Section 4 presents the numerical findings that validate the theoretical predictions.
\section{TRUNC finite element space}
Let $\Om\subset\mb{R}^d,d\geq2$ be a convex polytope. Let $\mc{T}_h$ be a simplicial triangulation of $\Omega$ with maximum mesh size $h$. We assume all elements in $\mc{T}_h$ are shape-regular in the sense of Ciarlet and Raviart~\cite{Ciarlet:1978}, i.e., there exists a constant $\gamma$ such that $h_K/\rho_K\le\gamma$, where $h_K$ is the diameter of the element $K$, and $\rho_K$ is the diameter of the largest ball inscribed into $K$, and $\gamma$ is the so-called chunkiness parameter ~\cite{BrennerScott:2008}. Let $\mc{F}_h$ and $\mc{V}_h$ be the sets of $(d-1)$-dimensional subsimplices and vertices in $\mc{T}_h$. We denote $\mc{F}_h^{\,B}=\set{F\in\mc{F}_h}{F\subset\pa\Om}$ and $\mc{F}_h^{\,I}=\mc{F}_h\setminus\mc{F}_h^{\,B}$ the set of boundary and interior subsimplex respectively. Similar notations apply to $\mc{V}_h$. Let $K\in\mc{T}_h$ be a $d$-dimensional simplex with vertices $\{a_i\}_{i=1}^{d+1}$. We denote by $F_i$ the $(d-1)$-dimensional subsimplex of $K$ opposite to vertex $a_i$, and by $e_{ij}$ the edge vector from $a_i$ to $a_j$, and by $\lam_i$ the barycentric coordinate associated with the vertex $a_i$. The TRUNC-type element is defined by the 
finite element triple $(K,Z_K,\Sigma_K)$ given by
\[
\left\{
\begin{aligned}
&Z_K=\mb{P}_2(K)\oplus\text{span}\set{\lam_i^2\lam_j-\lam_i\lam_j^2}{1\le i<j\le d+1},\\
&\Sigma_K=\set{p(a_i),(e_{ij}\cdot\na) p(a_i)}{1\le i\le d+1,j\neq i}.
\end{aligned}
\right.
\]
It is clear that the shape function space $Z_K$ is just the $d$-dimensional Zienkiewicz element. The following lemma gives the unisolvence of the finite element, we also list the explicit form of the basis functions.
\begin{lemma}
For the TRUNC-type element, the $\Sigma_K$ is $Z_K$-unisolvent. Moreover, let $\phi_i$ and $\phi_{ij}$ be the basis functions associated with the degree freedom $p(a_i)$ and $e_{ij}\cdot\na p(a_i)$, respectively. Then 
\begin{equation}\label{eq:basis}
\left\{
\begin{aligned}
\phi_i&=\lam_i+\sum_{j\neq i}(\lam_i^2\lam_j-\lam_i\lam_j^2),\\
\phi_{ij}&=\frac{1}{2}(\lam_i\lam_j+\lam_i^2\lam_j-\lam_i\lam_j^2).
\end{aligned}\right.
\end{equation}
\end{lemma}

The above lemma for $d=2$ may be found in~\cite{Ciarlet:1978}. 
\begin{proof}
It is clear that 
\[
\dim Z_k=C_{d+2}^2+C_{d+1}^2=(d+1)^2=\dim \Sigma_K.
\]
It suffices to show that a function $p\in Z_K$ vanishes if all the degrees of freedom are zeros. We employ the induction method. It is just the Zienkiewicz triangle when $d=2$, then we have $p\equiv0$. We assume that $p=0$ for $d=n-1$. When $d=n$, each face $F_i$ of $K$ is $(n-1)$-simplex, and there is a natural restriction of $(K,Z_K,\Sigma_K)$ which defines a finite element:
\[
(F_i,Z_{F_i},\Sigma_{F_i})=(K,Z_K,\Sigma_K)_{|_{F_i}}.
\]
It is clear that $(F_i,Z_{F_i},\Sigma_{F_i})$ is just $(n-1)$-dimensional TRUNC-type element on $F_i$. By the assumption, we have $p_{|_{F_i}}\equiv0$. We conclude that
\[
p=C\prod_{1\le i\le n+1}\lam_i,
\]
where $C$ is a constant polynomial factor. Since $p$ is cubic, then we have that $C=0$ and $p$ vanishes if all the degrees of freedom are zeros.

Next we verify the shape functions associated with the degree of freedoms. A direct calculations gives
\[
\left\{
\begin{aligned}
&e_{kl}\cdot\na\lam_i(a_k)=-\delta_{ik}+\delta_{il},\\
&e_{kl}\cdot\na\lam_i\lam_j(a_k)=\delta_{ik}\delta_{jl}+\delta_{il}\delta_{jk},\\
&e_{kl}\cdot\na(\lam_i^2\lam_j-\lam_i\lam_j^2)(a_k)=\delta_{ik}\delta_{jl}-\delta_{il}\delta_{jk},
\end{aligned}\right.
\]
where $1\le i\neq j\le d+1$ and $1\le k\neq l\le d+1$. Using the above identities, we have
\[
e_{kl}\cdot\na\phi_i(a_k)=\delta_{ik}\Lr{-1+\sum_{j\ne i}\delta_{jl}}+\delta_{il}\Lr{1-\sum_{j\ne i}\delta_{jk}}=0
\]
It immediately implies that $\phi_i$ is the basis function associated with the degree of freedom $p(a_i)$. Proceeding along the same line, we may verify the remaining basis functions. This completes the proof.
\end{proof}

The TRUNC-type finite element space is defined by
\[
Z_h=\set{v\in H^1(\Om)}{v_{|_K}\in Z_K,K\in\mc{T}_h;\,v(a),\na v(a)\,\text{are continuous for all}\, a\in\mc{V}_h},
\]
and the corresponding homogeneous finite element space is defined by
\[
V_h=\set{v\in Z_h}{v(a),\na v(a)\,\text{vanish for all}\,a\in\mc{V}_h^B}.
\]
It is clear that $V_h$ is subspace of $Z_h$. 

By~\eqref{eq:basis}, for any $v\in Z_h$ and $K\in\mc{T}_h$, we may write it as
\[
v_{|_K}=\sum_{1\le i\le d+1}v(a_i)\phi_i+\sum_{1\le i\le d+1}\sum_{j\ne i}e_{ij}\cdot\na v(a_i)\phi_{ij}.
\]
We define a local projection $\Pi_K{:} Z_h\rightarrow\mb{P}_2(K)$ by
\begin{equation}\label{eq:projection}
\Pi_K(v){:}=\sum_{1\le i\le d+1}v(a_i)\bar{\phi}_i+\sum_{1\le i\le d+1}\sum_{j\ne i}e_{ij}\cdot\na v(a_i)\bar{\phi}_{ij},
\end{equation}
where 
\[
\bar{\phi}_i=\lam_i\quad\text{and}\quad
\bar{\phi}_{ij}=\frac{1}{2}\lam_i\lam_j.
\]
It is clear the $\Pi_Kv$ is just a local quadratic interpolant of $v$ on element $K$. We denote the residue $\Pi_K^c=I-\Pi_K$. We define global projection $\Pi_h{:} Z_h\rightarrow H^1(\Om)$ with $\lr{\Pi_h}_{|_K}=\Pi_K$ for any $K\in\mc{T}_h$ and let 
the residue $\Pi_h^c=I-\Pi_h$ and we summarize the properties of $\Pi_K$ in the following theorem.
\begin{theorem}
The projection $\Pi_K{:}\,Z_h\rightarrow\mb{P}_2(K)$ has the following properties:
\begin{enumerate}
\item General weak continuity: For any $v\in Z_h$ and for any symmetric matrix $S$ of order $d$, we have
\begin{equation}\label{eq:weakcontinuity}
\sum_{1\le i\le d+1}(n_{F_i})^{\,\T}\cdot S\cdot Q_{F_i}\lr{\na\Pi_K^cv}=0,
\end{equation}
%where the $\mc{S}$ is a $d\times d$ symmetric matrix, and the $F_i$ is the $(d-1)$-dimensional subsimplex opposite to the vertex $a_i$, and the $n_{F_i}$ is normal unit vector out of $F_i$, and
where
\[
Q_{F_i}(v)=\frac{\abs{F_i}}{d}\sum_{j\neq i}v(a_j)
\]
is the linear quadrature scheme.%, and the $\abs{F_i}$ is the measure of $F_i$.

\item For any $v\in Z_h$, there holds 
\begin{equation}\label{interpolation:1}
\nm{\na^k\lr{v-\Pi_Kv}}{L^2(K)}\leq Ch^{m-k}\nm{\na^m v}{L^2(K)},\quad 0\le k,m\le3.
\end{equation}
\end{enumerate}
\end{theorem}

For any $F\in\pa K$ and $v\in Z_h$, $\lr{\na\Pi_Kv}_{|_F}\in\mb{P}_1(F)$, by the generalized weak continuity condition~\eqref{eq:weakcontinuity}, there holds
\begin{align*}
\int_{\pa K}n^{\T}\cdot S\cdot\na\Pi_Kv\md S&=\sum_{F\subset\pa K}n^{\T}\cdot S\cdot Q_F(\na\Pi_Kv)\\
&=\sum_{F\subset\pa K}n^{\T}\cdot S\cdot \lr{Q_F(\na v)-Q_F(\na\Pi_K^cv)}\\
&=\sum_{F\subset\pa K}n^{\T}\cdot S\cdot Q_F(\na v).
\end{align*}
Since $Q_F(\na v)$ is continuous across $F$, i.e., $\jump{Q_F(\na v)}_{|_F}=0$ for all $F\in\mc{F}_h^I$, the identity~\eqref{eq:weakcontinuity} may be regarded as a kind of general weak continuity. Moreover, it plays an important role in error estimates. \textsc{Shi}~\cite[identities (18)-(19)]{Shi:1987} has proposed two identities for the plate bending problem, i.e., for any $v,w\in Z_h$, there holds
\begin{equation}\label{eq:1}
\sum_{F\subset\pa K} Q_F(\pa_n\Pi_K^cv)=0,
\end{equation}
and
\begin{equation}\label{eq:2}
\sum_{F\subset\pa K}\Lr{\Lr{\frac{\pa^2\Pi_Kw}{\pa n^2}}_{|_F}Q_F(\pa_n\Pi_K^cv)+\Lr{\frac{\pa^2\Pi_Kw}{\pa n\pa t}}_{|_F}Q_F(\pa_t\Pi_K^cv)}=0,
\end{equation}
where $n$ is unit outward normal of $F$, and $t$ is unit tangential vector along $F$. Both~\eqref{eq:1} and~\eqref{eq:2} follow from the identity~\eqref{eq:weakcontinuity} by specifying $S=I_{2\times2}$ and $S=\na^2\Pi_Kw$.

We are ready to prove the above theorem. It is worth noting that for any $v\in Z_h$, $v_{|_K}\in\mb{P}_3(K)$ and $\Pi_Kv$ is the quadratic part of $v_{|_K}$. The interpolate estimate~\eqref{interpolation:1} is derived by standard interpolation estimates and the inverse estimates. It remains to prove~\eqref{eq:weakcontinuity}. %The interpolate estimate~\eqref{interpolation:1} is quite standard because $\Pi_K$ preserves quadratical polynomials. 

\begin{proof}
By~\eqref{eq:basis} and~\eqref{eq:projection}, we have that for any $v\in Z_h$ and $K\in\mc{T}_h$, 
\[
\Pi_K^cv\in\mathrm{span}\set{\lam_i^2\lam_j-\lam_i\lam_j^2}{1\le i<j\le d+1}.
\]
We only need to prove that~\eqref{eq:weakcontinuity} is valid for all $\lam_j^2\lam_k-\lam_j\lam_k^2$. 

The unit normal vector out of the $F_i$ is
\[
n_{F_i}=-\frac{\na\lam_i}{\abs{\na\lam_i}}=-\frac{d!\abs{K}}{\abs{F_i}}\na\lam_i,
\]
where $\abs{\na\lam_i}$ is the length of the vector $\na\lam_i$. A straightforward calculation gives
\[
Q_{F_i}\lr{\na\lr{\lam_j^2\lam_k-\lam_j\lam_k^2}}=\frac{\abs{F_i}}{d}\Big(\na\lam_k\lr{1-\delta_{ij}}-\na\lam_j\lr{1-\delta_{ik}}\Big),\quad 1\le k\le d+1.
\]

Noting that $S$ is symmetric, and employing the above two identities, we obtain
\begin{align*}
\sum_{1\le i\le d+1}&(n_{F_i})^{\,\T}\cdot S\cdot Q_{F_i}(\na\lr{\lam_j^2\lam_k-\lam_j\lam_k^2})
=(d-1)!\abs{K}\Big(\na\lam_j^{\T}\cdot S\cdot\na\lam_j\\
&\quad-\na\lam_k^{\T}\cdot S\cdot\na\lam_k-\sum_{i\notin\{j,k\}}\na\lam_i^{\T}\cdot S\cdot\lr{\na\lam_k-\na\lam_j}\Big)
=0,
\end{align*}
where we have used the facts
\[
\sum_{i\notin\{j,k\}}\na\lam_i=-\lr{\na\lam_j+\na\lam_k}\quad\text{and}\quad
\na\lam_j^{\T}\cdot S\cdot\na\lam_k=\na\lam_k^{\T}\cdot S\cdot\na\lam_j.
\]
This proves~\eqref{eq:weakcontinuity}.
\end{proof}

The degrees of freedom of $V_h$ involve the first order derivatives at each vertices, the associated standard interpolant is not well-defined for the functions belonging to $H^2(\Om)$. Since the TRUNC-type element shares the same degrees of freedom with the new Zienkiewicz-type element~\cite{WangShiXu:2007}~\footnote{This element is a natural extension of the Specht triangle~\cite{Specht:1988} in three dimension, which may be dubbed as the Specht tetrahedron.}, we adopt the regularized interpolant constructed in~\cite{LiMingWang:2021}. 
\begin{lemma}\label{lema:inter}
There exists a regularized interpolant $I_h:H_0^2(\Om)\rightarrow V_h$, such that for any $v\in H^m(\Om)\cap H_0^2(\Om)$ with $1\le m\le 3$, there holds
\begin{equation}\label{interpolation:2}
\nm{\na_h^k(v-I_hv)}{L^2}\le Ch^{m-k}\nm{\na^m v}{L^2},\quad 0\le k\le m,
\end{equation}
where $\na_h^kv$ is defined in the piecewise manner, i.e., $\lr{\na^k_hv}_{|_K}{:}=\na^k\lr{v_{|_K}}$. 

Moreover, we denote $\bar{I}_h{:}=\Pi_h\circ I_h$, and $\bar{I}_h$ is a quasi-interpolant in the sense that for any $v\in H^m(\Om)\cap H_0^2(\Om)$ with $1\le m\le 3$, there holds
\begin{equation}\label{interpolation:3}
\nm{\na_h^k(v-\bar{I}_hv)}{L^2}\le Ch^{m-k}\nm{\na^m v}{L^2},\quad 0\le k\le m.
\end{equation}
\end{lemma}

The interpolation estimate~\eqref{interpolation:2} and~\eqref{interpolation:3} follow from the interpolation error estimates in~\cite[Theorem 3, interpolation error estimate (33)]{LiMingWang:2021}. Note that both $I_h$ and $\bar{I}_h$ are quasi-local 
in the sense that: under the same condition of Lemma~\ref{lema:inter}, there exists $C$ that depends only on the chunkiness parameter such that
\begin{equation}\label{eq:quasi-local1}
\nm{\na^k(v-I_hv)}{L^2(K)}\le Ch_K^{m-k}\nm{\na^m v}{L^2(\om_K)},\quad 0\le k\le m,\;1\le m\le 3,
\end{equation}
and
\begin{equation}\label{eq:quasi-local2}
\nm{\na^k(v-\bar{I}_hv)}{L^2(K)}\le Ch_K^{m-k}\nm{\na^m v}{L^2(\om_K)},\quad 0\le k\le m,\;1\le m\le 3,
\end{equation}
where $\om_K$ is the set of all elements in $\mc{T}_h$ that have a nonempty intersection of $K$. We shall frequently use these facts later on.
\begin{proof}
Summing up the local interpolation estimate ~\eqref{interpolation:1} for all $K\in\mc{T}_h$, we have that for any $v\in V_h$, there holds
\begin{equation}\label{interpolation:1_1}
\nm{\na^k_h\lr{v-\Pi_hv}}{L^2}\leq Ch^{m-k}\nm{\na^m_hv}{L^2},\quad 0\le k,m\le3.
\end{equation}
Using the triangle inequality
\[
\nm{\na_h^k(v-\bar{I}_hv)}{L^2}\le\nm{\na_h^k(v-I_hv)}{L^2}+\nm{\na_h^k(I-\Pi_h)I_hv}{L^2},
\]
the estimates~\eqref{interpolation:1_1} and~\eqref{interpolation:2}, we obtain~\eqref{interpolation:3}.
\end{proof}
We shall frequently use the following trace inequalities.
\begin{lemma}
Given a shape-regular mesh $\mc{T}_h$, there exists a constant $C$ independent of $h_K$, but depends on $\gamma$
such that for any $K\in\mc{T}_h$ and $v\in H^1(K)$
\begin{equation}\label{ieq:trace_2}
\nm{v}{L^2(\pa K)}\le C\Lr{h_K^{-1/2}\nm{v}{L^2(K)}+\nm{v}{L^2(K)}^{1/2}\nm{\na v}{L^2(K)}^{1/2}}.
\end{equation}

For any polynomial $v\in\mb{P}_m(K)$ with $m\in\mathbb{N}\cup \{0\}$, then there exists a constant $C$ independent of $h_K$, but depends on the chunkiness parameter $\gamma$ such that
\begin{equation}\label{ieq:trace_3}
\nm{v}{L^2(\pa K)}\le Ch_K^{-1/2}\nm{v}{L^2(K)}.
\end{equation}
\end{lemma}
The trace inequality~\eqref{ieq:trace_2} differs from the standard trace inequality, which may be found in~\cite{BrennerScott:2008}, while~\eqref{ieq:trace_3} follows from~\eqref{ieq:trace_2} with the aid of the inverse inequality.
\section{TRUNC-type finite element approximation}
Without loss of generality, we assume that $f\in L^2(\Om)$. The weak form of ~\eqref{eq:bc} is to find $u\in H_0^2(\Om)$ such that
\begin{equation}\label{eq:variation}
a_{\veps}(u,v)=(f,v)\quad\text{for all\;} v\in H_0^2(\Om),
\end{equation}
where for any $v,w\in H_0^2(\Om)$,
\[
a_{\veps}(v,w){:}=\veps^2(\na^2v,\na^2w)+(\na v,\na w).
\]

We start with the following priori estimates for $u$.
\begin{theorem}\label{thm:reg}
Let $u\in H_0^2(\Om)$ be the weak solution of~\eqref{eq:variation} and $\Om$ be a bounded Lipschitz domain. Let
$\bar{u}\in H_0^1(\Om)$ be the weak solution of $-\triangle\bar{u}=f$. Suppose that $\bar{u}\in H^2(\Omega)$. 
%$\bar{u}\in H_0^1(\Om)\cap H^2(\Om)$ satisfies \[(\na\bar{u},\na v)=(f,v)\quad\text{for all\quad}v\in H_0^1(\Om).\]
Then there exists $C$ depending on $\Om$ but independent of $\veps$ such that
\begin{equation}\label{regularity:2}
\nm{\na^k(u-\bar{u})}{L^2}\le C\veps^{3/2-k}\nm{\na\bar{u}}{H^1}\qquad k=1,2.
\end{equation}

If $\Om$ is a convex polytope, then 
\begin{equation}\label{regularity:1}
\nm{\na^ku}{L^2}\le C\veps^{3/2-k}\nm{f}{L^2}\qquad k=2,3.
\end{equation}
%where $\bar{u}$ satisfies \[(\na\bar{u},\na v)=(f,v)\quad\forall v\in H_0^1(\Om).\]
\end{theorem}

The estimate~\eqref{regularity:1} indicates that the solution of~\eqref{eq:bc} exhibits a sharp boundary layer. The validity of ~\eqref{regularity:2} with $k=1$ and~\eqref{regularity:1} has been established in~\cite[Lemma 5.1]{Tai:2001} when $\Omega$ is a convex polygon. In~\cite{LiMingWang:2021}, the authors demonstrated the validity of~\eqref{regularity:2} and~\eqref{regularity:1} for a strain gradient elasticity model, which may be viewed as a tensorized version of ~\eqref{eq:bc}. Their proof applies to the convex polytope under a technical assumption (Assumption 1 therein), which has yet to be fully justified. We aim to establish the aforementioned result by integrating the approach from~\cite{LiMingWang:2021} with insights inspired by~\cite{Niu:2022}.

We start with the following auxiliary results. For $D$ a bounded Lipschitz domain in $\mb{R}^d$, let $0<\eta<c_0\text{diam}(D)$ and
\[
D_{\eta}=\set{x\in D}{\text{dist}(x,\pa D)<\eta}.
\]
\begin{lemma}
Let $D$ be a bounded Lipschitz domain in $\mb{R}^d$. There exists $C$ depending on $d$ and $D$ such that
\begin{align}
\nm{u}{L^2(D_{\eta})}&\le C\eta\nm{\na u}{L^2(D_{2\eta})}\qquad\text{for\quad}u\in H_0^1(D),\label{eq:small-ineq1}\\
\nm{u}{L^2(D_{\eta})}&\le C\eta^{1/2}\nm{u}{L^2(D)}^{1/2}\nm{u}{H^1(D)}^{1/2}
\qquad\text{for\quad}u\in H^1(D),\label{eq:small-ineq2}
\end{align}
and for $u\in H^2(D)\cap H_0^1(D)$, 
\begin{equation}\label{eq:small-ineq3}
\nm{u}{L^2(D_{\eta})}\le C\eta^{3/2}\nm{u}{L^2(D)}^{1/2}\nm{u}{H^2(D)}^{1/2}.
\end{equation}
\end{lemma}

The above lemma may be found in~\cite[Lemma 2.8]{Niu:2022}.%Shen:2018}.

\noindent\vskip .3cm\textbf{Proof of Theorem~\ref{thm:reg}\;}
Let $\rho_{\veps}\in C_0^\infty(\Om)$ be a cut-off function satisfying $0\le\rho\le 1$ and
\[
\rho_{\veps}=\left\{\begin{aligned}
0\quad&x\in\Om_{\veps},\\
1\quad&x\in\Om\backslash\Om_{2\veps},
\end{aligned}\right.
\]
and $\abs{\na^k\rho_{\veps}(x)}\le C\veps^{-k}$ for $k=1,2$ and $x\in\Om$.

Dfine $w{:}=u-\bar{u}\rho_{\veps}=u-\bar{u}+\bar{u}(1-\rho_{\veps})$. It is clear $w\in H_0^2(\Om)$. Note that
 \[
 a_{\veps}(u,w)=(f,w)=a_0(\bar{u},w).
 \]
 Hence,
\begin{equation} \label{eq:basic}
\begin{aligned}
 a_{\veps}(w,w)&=a_{\veps}(u,w)-a_{\veps}(\bar{u},w)+a_{\veps}(\bar{u}(1-\rho_{\veps}),w)\\
 &=(a_0-a_{\veps})(\bar{u},w)+a_{\veps}(\bar{u}(1-\rho_{\veps}),w)\\
 &=-\veps^2(\na^2\bar{u},\na^2w)+a_{\veps}(\bar{u}(1-\rho_{\veps}),w).
 \end{aligned}
 \end{equation}
 
Using the identity
 \[
 \na[\bar{u}(1-\rho_{\veps})]=(1-\rho_{\veps})\na\bar{u}-\bar{u}\na\rho_{\veps},
 \]
the estimates~\eqref{eq:small-ineq2} and~\eqref{eq:small-ineq3}, we obtain
 \begin{align}\label{eq:step1}
\nm{\na[\bar{u}(1-\rho_{\veps})]}{L^2(\Om)}&\le
C\Lr{\nm{\na\bar{u}}{L^2(\Om_{2\veps})}+\veps^{-1}\nm{\bar{u}}{L^2(\Om_{2\veps})}}\nn\\
&\le C\veps^{1/2}\nm{\bar{u}}{H^2(\Om)}.
\end{align}
 
Proceeding along the same line, we have
 \[
 \na^2[\bar{u}(1-\rho_{\veps})]=(1-\rho_{\veps})\na^2\bar{u}-\na\bar{u}\lr{\na\rho_{\veps}}^{\top}-\na\rho_{\veps}\lr{\na\bar{u}}^{\top}-\bar{u}\na^2\rho_{\veps}.
 \]
Invoking the estimates~\eqref{eq:small-ineq2} and~\eqref{eq:small-ineq3} again, we obtain
 \begin{align}\label{eq:step2}
\nm{\na^2[\bar{u}(1-\rho_{\veps})]}{L^2(\Om)}&\le\nm{\na^2\bar{u}}{L^2(\Om_{2\veps})}+
C\Lr{\veps^{-1}\nm{\na\bar{u}}{L^2(\Om_{2\veps})}+\veps^{-2}\nm{\bar{u}}{L^2(\Om_{2\veps})}}\nn\\
&\le \nm{\na^2\bar{u}}{L^2(\Om)}+C\veps^{-1/2}\nm{\bar{u}}{H^2(\Om)}.
\end{align}

 \iffalse
 \begin{align*}
\veps^2\abs{(\na^2(\bar{u}\rho_{\veps}),\na^2w)}&\le\dfrac{\veps^2}{2}\nm{\na^2 w}{L^2(\Om)}^2
+C\veps^2\nm{\na^2\bar{u}}{L^2(\Om)}\\
&\quad+C\nm{\na\bar{u}}{L^2(\Om_{2\veps})}^2+C\veps^{-2}\nm{\bar{u}}{L^2(\Om_{2\veps})}^2\\
&\le\dfrac{\veps^2}{2}\nm{\na^2 w}{L^2(\Om)}^2+C\veps\nm{\na\bar{u}}{H^1(\Om)}^2.
\end{align*} \fi

By~\eqref{eq:basic}, and using the Cauchy-Schwartz inequality, we obtain
\begin{align*}
a_{\veps}(w,w)
&\le\veps^2\nm{\na^2\bar{u}}{L^2(\Om)}\nm{\na^2w}{L^2(\Om)}\\
&\quad+\lr{\veps\nm{\na^2[\bar{u}(1-\rho_{\veps})]}{L^2(\Om)}+\nm{\na[\bar{u}(1-\rho_{\veps})]}{L^2(\Om)}}\\
&\qquad\times\lr{\veps\nm{\na^2w}{L^2(\Om)}+\nm{\na w}{L^2(\Om)}}.
\end{align*}
Substituting ~\eqref{eq:step1} and~\eqref{eq:step2} into the above inequality, using the coercivity of $a_{\veps}(\cdot,\cdot)$, we obtain
\[
\veps\nm{\na^2w}{L^2(\Om)}+\nm{\na w}{L^2(\Om)}\le C\veps^{1/2}\nm{\bar{u}}{H^2(\Om)},
\]
which together with~\eqref{eq:step1} and~\eqref{eq:step2} gives~\eqref{regularity:2} with $k=1$ and $k=2$, respectively.

If $\Om$ is a convex polytope, then we have the shift estimate for $\bar{u}$: 
\[
\nm{\bar{u}}{H^2(\Om)}\le C\nm{f}{L^2(\Om)},
\]
which together with the triangle inequality yields
\begin{align*}
\nm{\na^2u}{L^2(\Om)}&\le\nm{\na^2(u-\bar{u})}{L^2(\Om)}+\nm{\na^2\bar{u}}{L^2(\Om)}\\
&\le C\veps^{-1/2}\nm{\bar{u}}{H^2(\Om)}+\nm{\na^2\bar{u}}{L^2(\Om)}\\
&\le C\veps^{-1/2}\nm{f}{L^2(\Om)}.
\end{align*}
This proves~\eqref{regularity:1} with $k=2$.

If $\Om$ is a convex polytope, then we apply the shift estimate~\cite[Theorem 4.3.10, \S 4]{Mazya:2010} for the weak solution of
\[
\triangle^2u=\veps^{-2}\triangle(u-\bar{u})\quad\text{in\;}\Om\qquad u=\partial_n u=0\quad\text{on\;}\Om,
\]
and obtain
\[
\nm{u}{H^3(\Om)}\le C\veps^{-2}\nm{\triangle(u-\bar{u})}{H^{-1}(\Om)}\le C\veps^{-2}\nm{\na(u-\bar{u})}{L^2(\Om)},
\]
which together with~\eqref{regularity:2} with $k=1$ and the above shift estimate gives~\eqref{regularity:1} with $k=3$.
\qed

As a direct consequence of the above theorem, we have
\begin{coro}
%Under the same assumption of Theorem~\ref{thm:reg}, there holds
Let $u$ be the weak solution of~\eqref{eq:variation}. If $\Om$ is a convex polytope, then
\begin{equation}\label{eq:interpriori1}
\nm{\na^2u}{L^2}^{1/2}\nm{\na^3 u}{L^2}^{1/2}\le C\veps^{-1}\nm{f}{L^2},
\end{equation}
and
\begin{equation}\label{eq:interpriori2}
\nm{\na u}{L^2}^{1/2}\nm{\na^2 u}{L^2}^{1/2}\le C\nm{f}{L^2},
\end{equation}
\end{coro}
\begin{proof}
Combining the estimates $k=2$ and $k=3$ in~\eqref{regularity:1}, we obtain~\eqref{eq:interpriori1}.

A combination of the estimates $k=1$ and $k=2$ in~\eqref{regularity:2}, we obtain the following multiplicative priori estimate for $u$:
\[
\nm{\na(u-\bar{u})}{L^2}^{1/2}\nm{\na^2(u-\bar{u})}{L^2}^{1/2}\le C\nm{f}{L^2},
\]
which together with the a-priori estimate for $\bar{u}$ as
\begin{equation}\label{eq:possionest}
\nm{\na\bar{u}}{L^2}^{1/2}\nm{\na^2\bar{u}}{L^2}^{1/2}\le C\nm{f}{L^2}
\end{equation}
gives~\eqref{eq:interpriori2}.
\end{proof}

The finite element approximation of~\eqref{eq:variation} is to find $u\in V_h$ such that
\[
a_h(u_h,v)=(f,v)\quad\text{for all\;} v\in V_h,
\]
where for $v,w\in V_h$,
\[
a_h(v,w)=\veps^2(\na_h^2v,\na_h^2w)+(\na v,\na w).
\]
The corresponding energy norm is defined by 
\[
\wnm{v}{:}=\Lr{\veps^2\nm{\na_h^2v}{L^2}^2+\nm{\na_hv}{L^2}^2}^{1/2}\quad\text{for all $v\in V_h$}.
\]

It is well known that the above finite element approximation converges to the true solution only for special meshes in 
two dimensions. To get the element converging for arbitrary meshes, the variation formulation should be modified by numerical integration. Noting that
\begin{align*}
\lr{\na^2_hu,\na^2_hv}&=\lr{\na^2_h\Pi_hu,\na^2_h\Pi_hv}+\lr{\na^2_h\Pi_h^cu,\na^2_h\Pi_hv}+\lr{\na^2_h\Pi_hu,\na^2_h\Pi_h^cv}\\
&\qquad+\lr{\na^2_h\Pi_h^cu,\na^2_h\Pi_h^cv},
\end{align*}

and neglecting the two mixing terms $\lr{\na^2_h\Pi_h^cu,\na^2_h\Pi_hv}$ and $\lr{\na^2_h\Pi_hu,\na^2_h\Pi_h^cv}$. The TRUNC finite element approximation is to find $u_h\in V_h$ such that
\begin{equation}\label{eq:trunc}
b_h(u_h,v)=(f,v)\quad\text{for all\;} v\in V_h,
\end{equation}
where for $v,w\in V_h$,
\[
b_h(v,w){:}=\veps^2\lr{\lr{\na^2_h\Pi_hv,\na^2_h\Pi_hw}+\lr{\na^2_h\Pi_h^cv,\na^2_h\Pi_h^cw}}+\lr{\na v,\na w}.
\]

The following lemma shows the modified bilinear form $b_h$ is coercive on $V_h$. It immediately implies that~\eqref{eq:trunc} is uniquely solvable.
\begin{lemma}
The bilinear form $b_h$ is $V_h$-elliptic, i.e., for any $v\in V_h$, there holds
\begin{equation}\label{coercivity}
\frac{1}{2}\wnm{v}^2\le b_h(v,v).
\end{equation}
\end{lemma}
\begin{proof}
For any $v\in V_h$, it follows from
\[
2\lr{\na^2_h\Pi_hv,\na^2_h\Pi_h^cv}\le\nm{\na^2_h\Pi_hv}{L^2}^2+\nm{\na^2_h\Pi_h^cv}{L^2}^2
\]
that
\[
\nm{\na_h^2v}{L^2}^2\le2\lr{\nm{\na^2_h\Pi_hv}{L^2}^2+\nm{\na^2_h\Pi_h^cv}{L^2}^2}.
\]
This immediately implies~\eqref{coercivity}.
\end{proof}

It follows from the above coercivity inequality the following error bound. 
\begin{lemma}
Let $u$ and $u_h$ be the solutions of Problem~\eqref{eq:variation} and Problem ~\eqref{eq:trunc}, respectively, then
there exists $C$ such that
\begin{equation}\label{estimate:1}
\wnm{u-u_h}\le C\Lr{\wnm{u-I_hu}+\wnm{u-\bar{I}_h u}+\sup_{w\in V_h}\frac{E_h(u,\ov{I}_hu,w)}{\wnm{w}}},
\end{equation}
where the consistency error functional
\[
E_h(u,\ov{I}_hu,w){:}=\sum_{K\in\mc{T}_h}\int_{\pa K}n^{\T}\cdot\lr{\na^2u\cdot\na w-\na^2\bar{I}_hu\cdot\na\Pi_h^cw}\md S.
\]
\end{lemma}

The structure of the above lemma differs from the standard error estimate of the nonconforming method because there are two
approximation terms. See; cf.~\cite{Berge:1972}.

\begin{proof}
For any $v\in V_h$, we denote $w=v-u_h$. It is clear that
\begin{align*}
b_h(w,w)&=b_h(v,w)-b_h(u_h,w)=b_h(v,w)-(f,w)\\
&=(b_h-a_h)(v,w)+a_h(v-u,w)+a_h(u,w)-(f,w).
\end{align*}

Noting that $w\in H_0^1(\Om)$ and $w_{|_K}\in H^2(K)$ for all $K\in\mc{T}_h$, integration by parts, we obtain
\begin{align*}
(f,w)&=\sum_{K\in\mc{T}_h}(f,w)_{|_K}=\sum_{K\in\mc{T}_h}(\veps^2\triangle^2u-\triangle u,w)_{|_K}\\
&=a_h(u,w)-\sum_{K\in\mc{T}_h}\veps^2\int_{\pa K}n^{\T}\cdot\na^2u\cdot\na w\md S,
\end{align*}
where 
\(
n^{\T}\cdot\na^2u\cdot\na w=\sum_{1\le i,j\le d}n_i\pa_{ij}^2u\pa_jw.
\)

Using the fact $\lr{\Pi_hv}_{|_K}\in\mb{P}_2(K)$ and integration by parts, we obtain
\begin{align*}
\lr{\na_h^2\Pi_hv,\na^2_h\Pi_h^cw}
&=\lr{\na_h\triangle_h\Pi_hv,\na_h\Pi_h^cw}+\sum_{K\in\mc{T}_h}\int_{\pa K}n^{\T}\cdot\na^2\Pi_Kv\cdot\na\Pi_K^cw\md S\\
&=\sum_{K\in\mc{T}_h}\int_{\pa K}n^{\T}\cdot\na^2\Pi_Kv\cdot\na_h\Pi_K^cw\md S.
\end{align*}

A combination of the above three equations implies
\begin{align*}
b_h(w,w)&=a_h(v-u,w)-\veps^2(\na^2_h\Pi_h^cv,\na^2_h\Pi_hw)\\
&\qquad+\sum_{K\in\mc{T}_h}\veps^2\int_{\pa K}n^{\T}\cdot\Lr{\na^2u\cdot\na w-\na^2\Pi_Kv\cdot\na\Pi_K^cw}\md S.
\end{align*}

Using the coercivity inequality~\eqref{coercivity} and noting 
\[
\nm{\na_h^2\Pi_h^cv}{L^2}\le\nm{\na_h^2\lr{v-u}}{L^2}+\nm{\na_h^2\lr{u-\Pi_hv}}{L^2},
\]
we have
\[
\wnm{w}\le C\Lr{\wnm{u-v}+\wnm{u-\Pi_hv}+\sup_{w\in V_h}\frac{E_h(u,\Pi_hv,w)}{\wnm{w}}},
\]
Taking $v=I_hu$ in the right hand side of the above inequality, we obtain~\eqref{estimate:1}.
\end{proof}
We are ready to derive the convergence rate of the TRUNC element. The first step is estimating the approximation error in the energy norm.
\begin{lemma}
There holds
\begin{equation}\label{ieq:1}
\wnm{u-I_hu}\le Ch^{1/2}\nm{f}{L^2},
\end{equation}
and
\begin{equation}\label{ieq:2}
\wnm{u-\bar{I}_hu}\le Ch^{1/2}\nm{f}{L^2}.
\end{equation}
\end{lemma}

\begin{proof}
Using the interpolation estimate ~\eqref{interpolation:2}, we obtain
\begin{align*}
\nm{\na_h^2(u-I_hu)}{L^2}&=\nm{\na_h^2(u-I_hu)}{L^2}^{1/2}\nm{\na_h^2(u-I_hu)}{L^2}^{1/2}\\
&\le Ch^{1/2}\nm{\na^2u}{L^2}^{1/2}\nm{\na^3 u}{L^2}\\
&\le C\veps^{-1}h^{1/2}\nm{f}{L^2},
\end{align*}
where we have used the a priori estimates~\eqref{eq:interpriori1} in the last step.

Proceeding along the same line that leads to the above inequality, using~\eqref{eq:interpriori2} instead of~\eqref{eq:interpriori1}, we get
\[
\nm{\na(u-I_hu)}{L^2}\le Ch^{1/2}\nm{f}{L^2}.
\]
Combining the above two estimates, we obtain
\[
\wnm{u-I_hu}\le\veps\nm{\na_h^2(u-I_hu)}{L^2}+\nm{\na(u-I_hu)}{L^2}\le Ch^{1/2}\nm{f}{L^2}.
\]
This gives~\eqref{ieq:1}.

Proceeding along the same line that leads to~\eqref{ieq:1}, using the interpolation estimate~\eqref{interpolation:3} for $\bar{I}_h$, we obtain~\eqref{ieq:2}.
\end{proof}
\begin{theorem}\label{thm:main}
Let $u$ and $u_h$ be the solution of Problem~\eqref{eq:variation} and Problem ~\eqref{eq:trunc}, respectively. Then
\begin{equation}\label{estimate:2}
\wnm{u-u_h}\le Ch^{1/2}\nm{f}{L^2}.
\end{equation}
Moreover, if $u\in H^3(\Om)$, then
\begin{equation}\label{estimate:3}
\wnm{u-u_h}\le C\lr{\veps h+h^2}\nm{\na^3u}{L^2}.
\end{equation}
\end{theorem}
\begin{proof}
It remains to deal with the consistent error. We firstly write 
\begin{align*}
E_h(u,\bar{I}_hu,w)&=\sum_{K\in\mc{T}_h}\veps^2\int_{\pa K}n^{\T}\cdot\na^2_h\lr{u-\bar{I}_hu}\cdot\na_h\lr{w-\Pi_hw}\md S\\
&\qquad+\sum_{K\in\mc{T}_h}\veps^2\int_{\pa K}n^{\T}\cdot\na^2u\cdot\na_h\Pi_hw\md S\\
&={:}E_1+E_2.
\end{align*}

For any $K\in\mc{T}_h$, using~\eqref{eq:quasi-local2}, we get
\begin{align*}
\nm{\na^2(u-\bar{I}_hu)}{L^2(K)}
&=\nm{\na^2(u-\bar{I}_hu)}{L^2(K)}^{1/2}\nm{\na^2(u-\bar{I}_hu)}{L^2(K)}^{1/2}\\
&\le Ch_K^{1/2}\nm{\na^2u}{L^2(\om_K)}^{1/2}\nm{\na^3u}{L^2(\om_K)}^{1/2}.
\end{align*}
Using the trace inequality~\eqref{ieq:trace_2} and the boundedness of $\bar{I}_h$ with respect to $H^2$ norm and $H^3$ norm, we obtain
\begin{align*}
\nm{\na^2(u-&\bar{I}_hu)}{L^2(\pa K)}
\le C\bigg(h_K^{-1/2}\nm{\na^2(u-\bar{I}_hu)}{L^2(K)}\\
&\qquad+\nm{\na^2(u-\bar{I}_hu)}{L^2(K)}^{1/2}\nm{\na^3(u-\bar{I}_hu)}{L^2(K)}^{1/2}\bigg)\\
&\le C\Lr{h_K^{-1/2}\nm{\na^2(u-\bar{I}_hu)}{L^2(K)}
+\nm{\na^2u}{L^2(\om_K)}^{1/2}\nm{\na^3u}{L^2(\om_K)}^{1/2}}.
\end{align*}
A combination of the above two inequalities gives
\begin{equation}\label{eq:quasi-local3}
\nm{\na^2(u-\bar{I}_hu)}{L^2(\pa K)}\le C\nm{\na^2u}{L^2(\om_K)}^{1/2}\nm{\na^3u}{L^2(\om_K)}^{1/2}.
\end{equation}

Proceeding along the same line that leads to the above inequalities, using the trace inequality~\eqref{ieq:trace_3}, and the interpolation estimates~\eqref{interpolation:1}, we obtain
\begin{equation}\label{eq:tracepoly}
\nm{\na(w-\Pi_K w)}{L^2(\pa K)}
\le Ch_K^{-1/2}\nm{\na(w-\Pi_Kw)}{L^2(K)}
\le Ch_K^{1/2}\nm{\na^2w}{L^2(K)}.
\end{equation}
Combining the above two inequalities and using the regularity estimate~\eqref{eq:interpriori1}, we bound $E_1$ as
\begin{equation}\label{eq:5}
\begin{aligned}
E_1&\le\veps^2\sum_{K\in\mc{T}_h}\nm{\na^2(u-\bar{I}_hu)}{L^2(\pa K)}\nm{\na(w-\Pi_Kw)}{L^2(\pa K)}\\
&\le C\veps^2h^{1/2}\sum_{K\in\mc{T}_h}\nm{\na^2u}{L^2(\om_K)}^{1/2}\nm{\na^3 u}{L^2(\om_K)}^{1/2}
\nm{\na^2w}{L^2(K)}\\
&\le C\veps^2h^{1/2}\nm{\na^2u}{L^2}^{1/2}\nm{\na^3u}{L^2}^{1/2}\nm{\na_h^2w}{L^2}\\
&\le Ch^{1/2}\nm{f}{L^2}\wnm{w}.
\end{aligned}
\end{equation}

For any $\phi\in L^2(F)$, we define $P_0^F{:}\,L^2(F)\rightarrow\mb{P}_0(F)$ as the $L^2-$projection operator, i.e.,
\[
P_0^F(\phi){:}=\frac{1}{\abs{F}}\int_F\phi\md S.
\]
We denote the residue $R_0^F(\phi){:}=\phi-P_0^F(\phi)$.
It is clear that
\[
\int_F R_0^F(\phi)\md S=0.
\]
Similar notation applies to the $L^2-$projection operator on $L^2(K)$. Using the above identity, we write $E_2$ as
\begin{equation}\label{eq:6}
\begin{aligned}
E_2&=\veps^2\sum_{K\in\mc{T}_h}\sum_{F\subset\pa K}\int_Fn^{\T}\cdot R_0^F\lr{\na^2u}\cdot R_0^F\lr{\na_h\Pi_hw}\md S\\
&\qquad+\veps^2\sum_{K\in\mc{T}_h}\sum_{F\subset\pa K}\int_Fn^{\T}\cdot P_0^F\lr{\na^2u}\cdot \na_h\Pi_hw\md S\\
&={:}E_{21}+E_{22}.
\end{aligned}
\end{equation}

Proceeding along the same line that leads to~\eqref{eq:5}, we obtain
\begin{equation}\label{eq:7}
E_{21}\le\veps^2\sum_{K\in\mc{T}_h}\sum_{F\subset\pa K}\nm{R_0^F\lr{\na^2u}}{L^2(F)}\nm{R_0^F\lr{\na_h\Pi_hw}}{L^2(F)}
\le Ch^{1/2}\nm{f}{L^2}\wnm{w}.
\end{equation}

Noting that $\lr{\na_h\Pi_hw}|_F\in\mb{P}_1(F)$ and
\[
\int_F\na_h\Pi_hw\md S=Q_F(\na_h\Pi_hw)=Q_F(\na_hw)-Q_F\lr{\na_h\Pi_h^cw}.
\]
Using the above quadrature formula and the weak continuity~\eqref{eq:weakcontinuity}, we rewrite $E_{22}$ as
\begin{align*}
E_{22}&=\veps^2\sum_{F\in\mc{F}_h^I}n^{\T}\cdot P_0^F\lr{\na^2u}\cdot\jump{Q_F\lr{\na_hw}}+\veps^2\sum_{F\in\mc{F}_h^B}n^{\T}\cdot P_0^F\lr{\na^2u}\cdot Q_F\lr{\na_hw}\\
&\qquad-\veps^2\sum_{K\in\mc{T}_h}\sum_{F\subset\pa K}n^{\T}\cdot P_0^F\lr{\na^2u}\cdot Q_F\lr{\na_h\Pi_h^cw}\\
&=\veps^2\sum_{K\in\mc{T}_h}\sum_{F\subset\pa K}n^{\T}\cdot \lr{\na_h^2\bar{I}_hu-P_0^F\lr{\na^2u}}\cdot Q_F\lr{\na_h\Pi_h^cw},
\end{align*}
where we have used the facts that $\jump{Q_F(\na_hw)}_{|_F}=0$ for all $F\in\mc{F}_h^{\,I}$, and $Q_F(\na_hw)=0$ for all $F\in\mc{F}_h^B$ in the last step of the above derivation. 

Using a scaling argument, we obtain, there exists a constant $C$ independent of the size of $F$ such that for any polynomial $v$, there holds
\[
\abs{Q_F(v)}\le C\nm{v}{L^1(F)}.
\]
Hence, using the above inequality and the inverse inequality, we bound $E_{22}$ as
\begin{equation}\label{eq:8}
\begin{aligned}
E_{22}&\le C\veps^2\sum_{K\in\mc{T}_h}\sum_{F\subset\pa K}\nm{\na_h^2\bar{I}_hu-P_0^F\lr{\na^2u}}{L^{\infty}(F)}\nm{\na(w-\Pi_hw)}{L^1(F)}\\
&\le C\veps^2\sum_{K\in\mc{T}_h}\sum_{F\subset\pa K}\nm{\na_h^2\bar{I}_hu-P_0^F\lr{\na^2u}}{L^2(F)}
\nm{\na(w-\Pi_hw)}{L^2(F)}
\end{aligned}
\end{equation}
Using the triangle inequality, we obtain
\[
\nm{\na^2\bar{I}_hu-P_0^F\lr{\na^2u}}{L^2(F)}\le
\nm{\na^2(u-\bar{I}_hu)}{L^2(F)}+
\nm{\na^2u-P_0^F\lr{\na^2u}}{L^2(F)}.
\]
%For any $F\in\pa K$, and for any $v\in L^1(K)$, we denote by $\Pi_0: L^2(K)\to\mb{P}_0(K)$ as the $L^2-$projection operator, i.e,
%\[
%\Pi_0v=\dfrac{1}{\abs{K}}\int_K v(x)\md\,x.
%\]
Invoking the trace inequality~\eqref{ieq:trace_2} again, we obtain
\begin{align*}
\nm{&\na^2u-P_0^F\lr{\na^2u}}{L^2(F)}=\min_{c\in\mb{R}^{(d-1)\times (d-1)}}\nm{\na^2u-c}{L^2(F)}\\
&\le\nm{\na^2u-P_0^K(\na^2u)}{L^2(F)}\\
&\le C\Lr{h_K^{-1/2}\nm{\na^2u-P_0^K(\na^2u)}{L^2(K)}+\nm{\na^2u-P_0^K(\na^2u)}{L^2(K)}^{1/2}\nm{\na^3u}{L^2(K)}^{1/2}}\\
&\le 
C\nm{\na^2u}{L^2(K)}^{1/2}\nm{\na^3u}{L^2(K)}^{1/2},
\end{align*}
which together with~\eqref{eq:quasi-local3} leads to
\[
\nm{\na^2\bar{I}_hu-P_0^F\lr{\na^2u}}{L^2(F)}\le C\nm{\na^2u}{L^2(\om_K)}^{1/2}\nm{\na^3u}{L^2(\om_K)}^{1/2}.
\]

Substituting the above estimate into~\eqref{eq:8}, using~\eqref{eq:tracepoly}, we obtain
\begin{align*}
E_{22}&\le C\veps^2h^{1/2}\sum_{K\in\mc{T}_h}\nm{\na^2u}{L^2(\om_K)}^{1/2}\nm{\na^3u}{L^2(\om_K)}^{1/2}\nm{\na_h^2w}{L^2(K)}\\
&\le C\veps h^{1/2}\nm{\na^2u}{L^2}^{1/2}\nm{\na^3u}{L^2}^{1/2}\wnm{w}\\
&\le Ch^{1/2}\nm{f}{L^2}\wnm{w}.
\end{align*}

Combining the above inequalities~\eqref{eq:5}, ~\eqref{eq:6}, ~\eqref{eq:7} and~\eqref{eq:8}, we obtain
\begin{equation}\label{ieq:3}
\sup_{w\in V_h}\frac{E_h(u,\bar{I}_hu,w)}{\wnm{w}}\le Ch^{1/2}\nm{f}{L^2}.
\end{equation}
Substituting~\eqref{ieq:1},~\eqref{ieq:2} and ~\eqref{ieq:3} into~\eqref{estimate:1}, we obtain~\eqref{estimate:2}.

Using the interpolation estimates~\eqref{interpolation:2} and ~\eqref{interpolation:3}, we have
\[
\wnm{u-I_hu}+\wnm{u-\bar{I}_hu}\le C\lr{\veps h+h^2}\nm{\na^3u}{L^2}.
\]
Proceeding the same fashion to derive~\eqref{ieq:3}, we obtain
\[
E_h(u,\bar{I}_hu,w)\le C\veps^2h\nm{\na^3u}{L^2}\nm{\na_h^2w}{L^2}\le C\veps h\nm{\na^3u}{L^2}\wnm{w}.
\]
Combining the above estimates, we get~\eqref{estimate:3}.
\end{proof}
\section{Numerical Experiments}
This section focuses on evaluating the numerical accuracy of the TRUNC element in three dimensions. We consider the domain $\Omega=(0,1)^3$, and generate an initial mesh by dividing the unit cube into $64$ smaller cubes, with each small cube further divided into $6$ tetrahedra; we refer to Fig.\ref{mesh} for a plot. In all our tests, we measure the rates of convergence in the relative energy norm $\wnm{u-u_h}/\wnm{u}$ with respect to different values of $\veps$. To facilitate comparison, we also present the convergence rate of the NZT element.
\begin{figure}
\centering
\includegraphics[width=6cm, height=4.5cm]{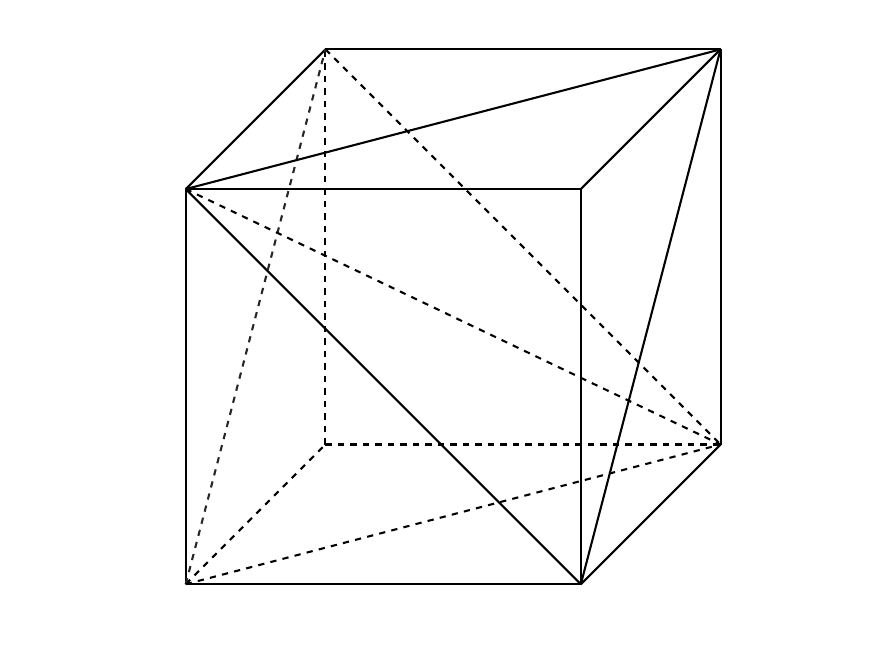}
\caption{Plots of meshes.}\label{mesh}
\end{figure}
\subsection{Example of smooth solution}
This test example evaluates the accuracy of the elements for a smooth solution defined as
\(
u=8\sin^2(\pi x_1)\sin^2(\pi x_2)\sin^2(\pi x_3).
\)
The source term $f$ is determined using~\eqref{eq:bc}. The convergence rates for the TRUNC element and the NZT element are reported in Table~\ref{tab:1} and Table~\ref{tab:2}, respectively. It is observed that the convergence rate appears to be linear when $\veps$ is large, transitioning to quadratic as $\veps$ approaches zero, aligning with the estimation in~\eqref{estimate:3}. Furthermore, the TRUNC element demonstrates higher accuracy compared to the NZT element. Although the NZT element in two dimensions (Specht triangle) is recognized as the optimal triangle plate bending element with $9$ degrees of freedom, as proven by ~\cite{ZienkiewiczTaylor:2009}, the NZT element remains highly efficient for solving the strain gradient elasticity model in three dimensions according to~\cite{LiMingWang:2021}.
\begin{table*}[htbp]
\caption{Convergence rates of the TRUNC element for a smooth solution.}\label{tab:1}
\begin{tabular*}{\textwidth}{@{\extracolsep\fill}lllllll@{}}
\hline\noalign{\smallskip}
$\veps\backslash h$ &$1/4$&$1/8$&$1/16$&$1/32$&$1/64$&$1/128$\\
\noalign{\smallskip}\hline\noalign{\smallskip}

1e-0&5.592e-01&3.016e-01&1.524e-01&7.626e-02&3.809e-02&1.897e-02\\
\noalign{\smallskip}
\textbf{rate}&&0.8908&0.9849&0.9988&1.0013&1.0057\\
\noalign{\smallskip}

1e-2&1.581e-01&4.302e-02&1.512e-02&6.563e-03&3.150e-03&1.561e-03\\
\noalign{\smallskip}
\textbf{rate}&&1.8782&1.5089&1.2036&1.0591&1.0127\\
\noalign{\smallskip}

1e-4&1.513e-01&3.491e-02&8.473e-03&2.079e-03&5.165e-04&1.299e-04\\
\noalign{\smallskip}
\textbf{rate}&&2.1158&2.0427&2.0271&2.0089&1.9910\\
\noalign{\smallskip}

1e-6&1.513e-01&3.491e-02&8.472e-03&2.078e-03&5.156e-04&1.286e-04\\
\noalign{\smallskip}
\textbf{rate}&&2.1158&2.0428&2.0276&2.0108&2.0040\\
\noalign{\smallskip}

\hline\noalign{\smallskip}
\end{tabular*}
\end{table*}
\begin{table*}[htbp]
\caption{Convergence rates of the NZT element for a smooth solution.}\label{tab:2}
\begin{tabular*}{\textwidth}{@{\extracolsep\fill}lllllll@{}}
\hline\noalign{\smallskip}
$\veps\backslash h$ &$1/4$&$1/8$&$1/16$&$1/32$&$1/64$&$1/128$\\
\noalign{\smallskip}\hline\noalign{\smallskip}

1e-0&6.913e-01&4.021e-01&1.956e-01&9.493e-02&4.670e-02&2.323e-02\\
\noalign{\smallskip}
\textbf{rate}&&0.7817&1.0399&1.0428&1.0236&1.0073\\
\noalign{\smallskip}

1e-2&2.294e-01&6.908e-02&2.241e-02&9.240e-03&4.338e-03&2.089e-03\\
\noalign{\smallskip}
\textbf{rate}&&1.7315&1.6242&1.2781&1.0908&1.0541\\
\noalign{\smallskip}

1e-4&1.939e-01&4.619e-02&1.104e-02&2.720e-03&6.784e-04&1.706e-04\\
\noalign{\smallskip}
\textbf{rate}&&2.0699&2.0644&2.0216&2.0032&1.9912\\
\noalign{\smallskip}

1e-6&1.939e-01&4.618e-02&1.104e-02&2.678e-03&6.631e-04&1.654e-04
\\
\noalign{\smallskip}
\textbf{rate}&&2.0700&2.0646&2.0436&2.0139&2.0031
\\
\noalign{\smallskip}

\hline\noalign{\smallskip}
\end{tabular*}
\end{table*}
\subsection{Example of solution with boundary layer}
In this example, we test the performance of both elements for a solution with a boundary layer. As in~\cite{LiMingWang:2021}, we construct a solution as
\[
u=\prod_{i=1}^3\lr{\exp(\sin\pi x_i)-1-\veps\varphi(x_i)}
\]
with
\[
\varphi(x)=\pi\dfrac{\cosh(1/2\veps)-\cosh\lr{(2x-1)/2\veps}}{\sinh(1/2\veps)}.
\]
A direct calculation gives
\[
\bar{u}=\lim_{\veps\rightarrow 0}u=\prod_{i=1}^3\lr{\exp(\sin\pi x_i)-1},
\]
with $\bar{u}|_{\pa\Om}=0$ and $\pa_n\bar{u}|_{\pa\Om}\neq0$. It is clear that the $\pa_n u$ has a boundary layer. The source term $f$ is also computed by~\eqref{eq:bc}. We report the rates of convergence for $\veps=10^{-6}$ in Table~\ref{tab:3}. A half order rate of convergence is observed, which is consistent with the theoretical prediction~\eqref{estimate:2}. We also observe that the TRUNC-type element is more accurate than the NZT element.
\begin{table*}[htbp]
\caption{Rate of convergence for a solution with boundary layer.}\label{tab:3}
\begin{tabular*}{\textwidth}{@{\extracolsep\fill}lllllll@{}}
\hline\noalign{\smallskip}
$h$ &$1/4$&$1/8$&$1/16$&$1/32$&$1/64$&$1/128$\\
\noalign{\smallskip}\hline\noalign{\smallskip}

\textbf{TRUNC}&2.654e-01&1.489e-01&9.996e-02&6.993e-02&4.930e-02&3.481e-02\\
\noalign{\smallskip}
\textbf{rate}&&0.8336&0.5754&0.5154&0.5043&0.5021\\
\noalign{\smallskip}
\hline

\textbf{NZT}&3.184e-01&1.854e-01&1.253e-01&8.757e-02&6.167e-02&4.356e-02\\
\noalign{\smallskip}
\textbf{rate}&&0.7807&0.5648&0.5170&0.5058&0.5014\\
\noalign{\smallskip}
\hline
\end{tabular*}
\end{table*}
\section{Conclusion}
We introduced the TRUNC finite element method applicable in any dimension and established the weak continuity condition, significantly streamlining the existing error estimate for the TRUNC approximation of the biharmonic problem. Leveraging this weak continuity condition, we derived a uniform error estimate for a modified Poisson equation posed on a convex polytope. Our numerical examples provide empirical support for the theoretical predictions.

\bibliographystyle{amsplain}
\bibliography{sg-8}
\end{document}